\documentclass[12pt]{amsart}
\usepackage{amsmath, amsthm, amssymb}
\title{On VC-minimal theories and variants}
\author{Vincent Guingona}
\address[Guingona]{University of Notre Dame \\ Department of Mathematics \\ 255 Hurley, Notre Dame, IN 46556}
\email{guingona.1@nd.edu}
\urladdr{http://www.nd.edu/~vguingon/}
\author{Michael C. Laskowski}
\address[Laskowski]{University of Maryland \\ Department of Mathematics \\ College Park, MD 20742}
\email{mcl@math.umd.edu}
\urladdr{http://www.math.umd.edu/~mcl/}
\thanks{\indent Both authors are partially supported by Laskowski's NSF grants DMS-0600217 and 0901336}
\date{\today}

\newtheorem{thm}{Theorem}[section]
\newtheorem{cor}[thm]{Corollary}
\newtheorem{lem}[thm]{Lemma}
\newtheorem{prop}[thm]{Proposition}

\theoremstyle{remark}

\newtheorem{expl}[thm]{Example}

\theoremstyle{definition}
\newtheorem{defn}[thm]{Definition}

\newcommand{\dom}{\mathrm{dom} }
\newcommand{\tp}{\mathrm{tp} }

\newcommand{\ID}{\mathrm{ID} }

\newcommand{\leng}{\mathrm{lg} }

\newcommand{\Th}{\mathrm{Th} }

\begin{document}

\begin{abstract}
 In this paper, we study VC-minimal theories and explore related concepts.  We first define the notion of convex orderablility and show that this lies strictly between VC-minimality and dp-minimality.  Next, we define the notion of weak VC-minimality, show it lies strictly between VC-minimality and dependence, and show that all unstable weakly VC-minimal theories interpret an infinite linear order.  Finally, we define the notion full VC-minimality, show that this lies strictly between weak o-minimality and VC-minimality, and show that theories that are fully VC-minimal have low VC-density.
\end{abstract}

\maketitle

\section{Introduction}\label{Section_intro}


Hans Adler introduced VC-minimal theories in \cite{Adler2008}.  This notion generalizes both the notion of weak o-minimality and the notion of strong minimality and encompasses algebraically closed valued fields.  In fact, like algebraically closed valued fields, definable subsets of the universe are disjoint unions of swiss chesses (see \cite{Adler2008} for a formal definition).  Unfortunately, VC-minimality can be a difficult notion to work with.  In this paper, we propose several variants of VC-minimality each aimed at patching a specific problem with the class of theories.

In Section \ref{Section_COS}, we study the notion of convex orderability.  This property fits properly between VC-minimality and dp-minimality and is, in some ways, more robust than VC-minimality.  For example, convex orderability is closed under reducts and amounts to a simple constraint on finite type spaces.  Additionally, convexly orderable theories ``look like'' weakly o-minimal theories in some sense.  In Section \ref{Section_VCVar}, we introduce two variants on VC-minimality.  First, we define a generalization of VC-minimality called weak VC-minimality and we show that all weakly VC-minimal theories satisfy the Kueker Conjecture.  Next, we introduce the stronger notion of full VC-minimality and show that a theory that is fully VC-minimal has low VC-density (i.e., all formulas $\varphi(\overline{x}; \overline{y})$ have VC-density $\le \leng(\overline{x})$).

For this paper, a ``formula'' will mean a $\emptyset$-definable formula in a fixed language $L$ unless otherwise specified.  If $\theta(\overline{x})$ is a formula, then let me denote $\theta(\overline{x})^0 = \neg \theta(\overline{x})$ and $\theta(\overline{x})^1 = \theta(\overline{x})$.  We will be working in a complete, first-order theory $T$ in a fixed language $L$ with monster model $\mathfrak{C}$.  Fix $M \models T$ (so $M \preceq \mathfrak{C}$) and a partitioned $L$-formula $\varphi(\overline{x}; \overline{y})$.  By $\varphi(M; \overline{b})$ for some $\overline{b} \in \mathfrak{C}^{\leng(\overline{y})}$, we mean the following subset of $M^{\leng(\overline{x})}$:

\[
 \varphi(M; \overline{b}) = \{ \overline{a} \in M^{\leng(\overline{x})} : \,\, \models \varphi(\overline{a}; \overline{b}) \}.
\]

First, consider any set $X$ and let $\mathcal{A} \subseteq \mathcal{P}(X)$.  We will say that \textit{$\mathcal{A}$ is independent with respect to $X$} if, for all functions $s : \mathcal{A} \rightarrow 2$, we have that $\bigcap_{B \in \mathcal{A}} B^{s(B)} \neq \emptyset$, where we let $B^0 = (X - B)$ and $B^1 = B$.  For some $N < \omega$, we will say that $\mathcal{A}$ has \textit{independence dimension $N$ with respect to $X$}, denoted $\ID_X(\mathcal{A}) = N$, if $N$ is maximal such that there exists $\mathcal{B} \subseteq \mathcal{A}$ with $|\mathcal{B}| = N$ and $\mathcal{B}$ is independent with respect to $X$.  We apply these general notions to the specific case of the formula $\varphi$ and subsets of $\mathfrak{C}^{\leng(\overline{y})}$ as follows: We will say that a set $B \subseteq \mathfrak{C}^{\leng(\overline{y})}$ is $\varphi$-\textit{independent} if the set of realizations $\mathcal{B} = \{ \varphi(\mathfrak{C}; \overline{b}) : \overline{b} \in B \}$ is independent with respect to $\mathfrak{C}^{\leng(\overline{x})}$ (i.e. for any map $s : B \rightarrow 2$, the set of formulas $\{ \varphi(\overline{x}; \overline{b})^{s(\overline{b})} : \overline{b} \in B \}$ is consistent).  We will say that $\varphi$ has \textit{independence dimension} $N < \omega$, which we will denote by $\ID(\varphi) = N$, if $\mathcal{A} = \{ \varphi(\mathfrak{C}; \overline{b}) : \overline{b} \in \mathfrak{C}^{\leng(\overline{y})} \}$ has independence dimension $N$ with respect to $\mathfrak{C}^{\leng(\overline{x})}$.  We will say that $\varphi$ is dependent (some authors call this NIP for ``not the independence property'') if $\ID(\varphi) = N$ for some $N < \omega$.  Finally, we will say that a theory $T$ is dependent if all partitioned formulas are dependent.

By a ``$\varphi$-type over $B$'' for some set $B$ of $\leng(\overline{y})$-tuples we mean a consistent set of formulas of the form $\varphi(\overline{x}; \overline{b})^t$ for some $t < 2$ and ranging over all $\overline{b} \in B$.  If $p$ is a $\varphi$-type over $B$, then we will say that $p$ has domain $\dom(p) = B$.  For any $B$ a set of $\leng(\overline{y})$-tuples, the space of all $\varphi$-types with domain $B$ is denoted $S_\varphi(B)$.


We conclude this section with a discussion of VC-minimal theories.  The following definition was given by Adler in \cite{Adler2008}:

\begin{defn}\label{Defn_VCMin}
 A set of formulas $\Phi = \{ \varphi_i(x; \overline{y}_i) : i \in I \}$ is a \emph{VC-instantiable family} if the set $\mathcal{A}_\Phi = \{ \varphi_i(\mathfrak{C}; \overline{b}) : i \in I, \overline{b} \in \mathfrak{C}^{\leng(\overline{y}_i)} \}$ has independence dimension $\le 1$ with respect to $\mathfrak{C}$.  An \emph{instance} of $\Phi$ means $\varphi_i(x; \overline{b})$ for some $i \in I$ and $\overline{b} \in \mathfrak{C}^{\leng(\overline{y}_i)}$.  We say that a theory $T$ is \emph{VC-minimal} if there exists a VC-minimal instantiable family $\Phi$ such that all parameter-definable formulas $\psi(x)$ are $T$-equivalent to a boolean combination of instances of $\Phi$.
\end{defn}

VC-minimal theories have a ``skeleton'' of formulas that have independence dimension $\le 1$.  VC-minimality was developed primarily to generalize the appearance of swiss cheeses in the theory of algebraically closed valued fields (see Proposition 7 of \cite{Adler2008}).  We get the following theorem relating VC-minimality to other model-theoretic properties:

\begin{thm}[\cite{Adler2008}]\label{Thm_VCMinEasy}
 The following hold:
 \begin{itemize}
  \item [(i)] If $T$ is strongly minimal, o-minimal, or weakly o-minimal, then $T$ is VC-minimal.
  \item [(ii)] The theory $\mathrm{ACVF}$ is VC-minimal.
  \item [(iii)] If $T$ is VC-minimal, then $T$ is dp-minimal.
 \end{itemize}
\end{thm}

A simple example shows that VC-minimal theories are not closed under reduct.  This example is given in \cite{Adler2008}.

\begin{expl}\label{Expl_VCReduct}
 One can check that the theory of a dense cyclic order is not VC-minimal.  However, this theory is the reduct of dense linear order via $R(x,y,z) = x < y < z \vee y < z < x \vee z < x < y$.  Therefore, since dense linear order is o-minimal, hence VC-minimal, this shows that VC-minimality is not closed under reducts.
\end{expl}


\section{Convexly Orderable Structures}\label{Section_COS}


The main objective of this section is to propose a robust surrogate for VC-minimality.  In a weakly o-minimal theory $T$, for any formula $\varphi(x; \overline{y})$, there exists $K < \omega$ such that, for all $\overline{b} \in \mathfrak{C}^{\leng(\overline{b})}$, $\varphi(\mathfrak{C}; \overline{b})$ is a union of at most $K$ convex subsets of $\mathfrak{C}$.  Thus, models of a weakly o-minimal theory are convexly orderable, defined as follows:

\begin{defn}\label{Defn_ConvexlyOrderable}
 Fix an $L$-structure $M$.  We say that $M$ is \emph{convexly orderable} if there exists a linear order $<$ on $M$ (not necessarily definable) such that, for all formulas $\varphi(x; \overline{y})$, there exists $K < \omega$ such that, for all $\overline{b} \in M^{\leng(\overline{b})}$, $\varphi(M; \overline{b})$ is a union of at most $K$ $<$-convex subsets of $M$.  We say a theory $T$ is \emph{convexly orderable} if all of its models are (equivalently, by Proposition \ref{Prop_ConvexOrderability}, if there exists a convexly orderable model of $T$).
\end{defn}

Convex orderability is a finitary property.

\begin{lem}\label{Lem_COFinitary}
 Suppose $M$ is an $L$-structure.  Suppose that, for each $L$-formula $\varphi(x; \overline{y})$, there exists $K_\varphi < \omega$ such that, for all finite collections of $M$-definable formulas $\{ \varphi_i(x; \overline{b}_i) : i < L \}$, there exists $<$ an ordering on $M$ such that, for all $i < L$, $\varphi_i(M; \overline{b}_i)$ is a union of at most $K_{\varphi_i}$ $<$-convex subsets of $M$.  Then, $M$ is convexly orderable. 
\end{lem}

\begin{proof}
 Let $\Theta(x)$ be the collection of all $M$-definable formulas $\theta(x)$, let $L^* = \{ P_\theta : \theta \in \Theta \}$ where $P_\theta$ is a unary relation symbol, and let $L^{**} = L^* \cup \{ < \}$.  Let $M^*$ be the obvious $L^*$ structure with universe $M$ (where $P_\theta(M^*) = \theta(M)$).  Consider a finite $\Theta_0 \subseteq \Theta$.  By hypothesis, there exists $<$ an ordering on $M^*$ so that, for all $\varphi(x; \overline{b}) \in \Theta_0$, $P_{\varphi(x; \overline{b})}(M^*)$ is a union of at most $K_\varphi$ $<$-convex subsets of $M^*$.  By compactness, there exists an $L^{**}$-structure $M^{**} \supseteq M^*$ so that, for all $\varphi(x; \overline{b}) \in \Theta$, $P_{\varphi(x; \overline{b})}(M^{**})$ is a union of at most $K_\varphi$ $<$-convex subsets of $M^{**}$.  However, since convexity is closed under linear subspace, if we take the suborder $< |_{M \times M}$ on $M$, we get the desired result.
\end{proof}

\begin{prop}\label{Prop_ConvexOrderability}
 If $M$ is a convexly orderable $L$-structure and $M \equiv N$, then $N$ is convexly orderable.
\end{prop}

\begin{proof}
 Let $<$ be a fixed ordering on $M$ given by convex orderability.  For each $L$-formula $\varphi(x; \overline{y})$, let $K_\varphi$ be such that, for all $\overline{b} \in M^{\leng(\overline{y})}$, $\varphi(M; \overline{b})$ is a union of at most $K_\varphi$ $<$-convex subsets of $M$.  We now use Lemma \ref{Lem_COFinitary} to conclude that $N$ is convexly orderable.

 Suppose $\Theta_0 = \{ \varphi_i(x; \overline{b}_i) : i < L \}$ and let $\eta_{\Theta_0} : {}^L 2 \rightarrow 2$ be such that, for all $s \in {}^L 2$:
 \[
  \eta_{\Theta_0}(s) = 1 \text{ if and only if } \models \exists x \bigwedge_{i < L} \varphi_i(x; \overline{b}_i)^{s(i)}.
 \]
 Then, since the following sentence holds in $N$, it also holds in $M$:
 \[
  \sigma_{\Theta_0} = \exists \overline{z}_0 ... \overline{z}_{L-1} \bigwedge_{s \in {}^L 2} \left( \exists x \bigwedge_{i < L} \varphi_i(x; \overline{z}_i)^{s(i)} \right)^{\eta_{\Theta_0}(s)}.
 \]
 Let $\overline{c}_0, ..., \overline{c}_{L_1}$ from $M$ be a witness to this.  Then by definition, $\varphi_i(M; \overline{c}_i)$ is the union of at most $K_{\varphi_i}$ $<$-convex subsets of $M$.  Finally, since the type spaces match, we can use this to put an ordering $<$ on $N$ so that $\varphi_i(N; \overline{b}_i)$ is the union of at most $K_{\varphi_i}$ $<$-convex subsets of $N$.  By Lemma \ref{Lem_COFinitary}, this gives us the desired conclusion.
\end{proof}

So convex orderability is really a constraint on the finite type spaces of a structure.  It is clear that if $T$ is a reduct of a convexly orderable theory, then $T$ is convexly orderable.  Therefore, Example \ref{Expl_VCReduct} shows there are theories that are convexly orderable but not VC-minimal.  The converse, however, is true:

\begin{thm}\label{Thm_VCMinCO}
 If $T$ is a VC-minimal theory, then $T$ is convexly orderable.
\end{thm}

In order to prove this, we actually show something more general about sets with independence dimension $\le 1$.

\begin{prop}\label{Prop_IDoneOrder}
 If $X$ is any set and $\mathcal{A} \subseteq \mathcal{P}(X)$ has independence dimension $\le 1$, then there exists $<$ a linear order on $X$ such that, for all $A \in \mathcal{A}$, either $A^0$ is a $<$-convex subset of $X$ or $A^1$ is a $<$-convex subset of $X$.
\end{prop}

First we show this for finite $\mathcal{A}$, then use compactness to yield the general result.

\begin{lem}\label{Lem_IDoneLO}
 Fix $X$ any set and let $\mathcal{A} \subseteq \mathcal{P}(X)$ be any finite set of subsets of $X$ that has independence dimension at most one.  Then, there exists a linear order $<$ on $X$ such that, for all $A \in \mathcal{A}$, either $A^0$ or $A^1$ is a $<$-convex subset of $X$.
\end{lem}

\begin{proof}[Proof of Lemma \ref{Lem_IDoneLO}]
 Notice that it suffices to show this for $\mathcal{A}$ such that $\emptyset \notin \mathcal{A}$ and $X \notin \mathcal{A}$ (regardless of the order we put on $X$, these two sets are convex).  We first prove the following claim via induction on the size of $\mathcal{A}$:

 \noindent \textbf{Claim 1.} If $\mathcal{A}$ is such that there exists no $A, B \in \mathcal{A}$ where $A \cap B \neq \emptyset$, $(B - A) \neq \emptyset$, and $(A - B) \neq \emptyset$, then there exists a linear order $<$ on $X$ such that, for all $A \in \mathcal{A}$, $A$ is a $<$-convex subset of $X$.

 \noindent \textit{Proof of Claim 1.}
  This proceeds by induction on $|\mathcal{A}|$.  If $|\mathcal{A}| = 1$, this is trivial, so suppose not.  Fix $A \in \mathcal{A}$ that is $\subseteq$-maximal.  As stated above, we may assume $A \neq X$.  Then, for all $B \in \mathcal{A}$, notice that either $B \subseteq A^0$ or $B \subseteq A^1$, but not both.  To see this, suppose that $A^0 \cap B^1 \neq \emptyset$ and $A^1 \cap B^1 \neq \emptyset$.  Further, suppose that $A^1 \cap B^0 = \emptyset$ (i.e., $A \subset B$).  This contradicts the maximality of $A$, so $A^1 \cap B^0 \neq \emptyset$.  Hence $A \cap B \neq \emptyset$, $(B - A) \neq \emptyset$, and $(A - B) \neq \emptyset$, contrary to assumption.

  Therefore, we can break up $\mathcal{A} - \{ A \}$ into two disjoint subsets:
  \[
   \mathcal{A}_t = \{ B \in \mathcal{A} : B \neq A \wedge B \subseteq A^t \}
  \]
  for each $t < 2$.  Now, as $\mathcal{A}_t$ has independence dimension $\le 1$ with respect to $A^t$, by induction hypothesis there exists an ordering $<_t$ on $A^t$ such that, for all $B \in \mathcal{A}_t$, $B$ is a $<_t$-convex subset of $A^t$.  Define $<$ globally by letting $<$ extend $<_0$ and $<_1$ and setting $A^0 < A^1$ (i.e., for all $a_0 \in A^0$ and $a_1 \in A^1$, $a_0 < a_1$).  Then, for all $B \in \mathcal{A}$, either:
  \begin{itemize}
   \item [(i)] $B = A^t$ for some $t < 2$ and $B$ is clearly $<$-convex;
   \item [(ii)] $B \in \mathcal{A}_0$ and $B$ is $<_0$-convex, hence since $B \subseteq A^0$ and $<$ extends $<_0$, $B$ is $<$-convex; or
   \item [(iii)] $B \in \mathcal{A}_1$ and $B$ is $<_1$-convex, hence since $B \subseteq A^1$ and $<$ extends $<_1$, $B$ is $<$-convex.
  \end{itemize}
  \hfill \small Claim 1 \normalsize $\Box$

 Now, more generally, fix $\mathcal{A} \subseteq \mathcal{P}(X)$ finite with $\ID_X(\mathcal{A}) \le 1$.  Fix any $A \in \mathcal{A}$ (recall that we may assume $A \neq X$ and $A \neq \emptyset$).  Now, by independence dimension $\le 1$, for any $B \in \mathcal{A}$, one of the following four conditions hold:

 \begin{itemize}
  \item [(i)] $B^0 \subseteq A^0$,
  \item [(ii)] $B^0 \subseteq A^1$,
  \item [(iii)] $B^1 \subseteq A^0$, or
  \item [(iv)] $B^1 \subseteq A^1$,
 \end{itemize}

 Let $\mathcal{A}_t = \{ B^s : B \in \mathcal{A} - \{ A \}, s < 2, B^s \subseteq A^t \}$ for each $t < 2$.  By the statement above, for any $B \in \mathcal{A} - \{ A \}$, there exists $t, s < 2$ such that $B^s \in \mathcal{A}_t$.  Furthermore, $\mathcal{A}_t$ has independence dimension $\le 1$ with respect to $A^t$ for both $t < 2$ (replacing $B$ with $B^0$ does not change the independence dimension).  We claim that $\mathcal{A}_t$ satisfies the hypotheses of Claim 1, hence showing that there is a linear order $<_t$ on $A^t$ so that, for all $B \in \mathcal{A}_t$, $B$ is a $<_t$-convex subset of $A^t$.

 To see this, choose any $B_0, B_1 \in \mathcal{A}_t$ and suppose that $B_0 \cap B_1 \neq \emptyset$, $(B_0 - B_1) \neq \emptyset$, and $(B_1 - B_0) \neq \emptyset$.  Then, since $A^t \neq X$ and $B_0, B_1 \subseteq A^t$, we have that $B_0 \cup B_1 \neq X$.  These four conditions together contradict the fact that $\ID_X(\mathcal{A}) \le 1$.

 Now, as before, let $<$ be the global ordering on $X$ that extends $<_0$ and $<_1$ so that $A^0 < A^1$.  Then, we claim that this ordering on $X$ satisfies the desired condition.  To see this, fix any $B \in \mathcal{A}$.  Then, for some $t, s < 2$, $B^s \in \mathcal{A}_t$ (hence $B^s \subseteq A^t$).  Therefore, $B^s$ is a $<_t$-convex subset of $A^t$.  As $<$ extends $<_t$ and $B^s \subseteq A^t$, $B^s$ is a $<$-convex subset of $X$.  That is, we have shown that all elements of $\mathcal{A}$ are either $<$-convex subsets of $X$ or the compliment of $<$-convex subsets of $X$.  This concludes the proof of the lemma.
\end{proof}

\begin{proof}[Proof of Proposition \ref{Prop_IDoneOrder}]
 As in the proof of Lemma \ref{Lem_COFinitary}, we use compactness.  Fix any set $X$ and fix $\mathcal{A} \subseteq \mathcal{P}(X)$ with independence dimension $\le 1$.  Consider the language $L$ which consists of unary predicate symbols $P_B$ for each $B \in \mathcal{A}$ and let $M = (X; B)_{B \in \mathcal{A}}$ be the natural $L$-structure.  Let $L' = L \cup \{ < \}$, where $<$ is a binary relation symbol.  Lemma \ref{Lem_IDoneLO}, there exists $N \succeq M$ such that $N$ is naturally an $L'$-structure with each $P_B(N)$ a $<$-convex subset of $N$ or the compliment of a $<$-convex subset of $N$.  However, being a convex subset is closed under linear subspace.  Thus, if we consider $< |_{X \times X}$, we get a linear ordering on $X$ with the desired result.
\end{proof}

We are now ready to prove Theorem \ref{Thm_VCMinCO}, showing that a structure with a VC-minimal theory is convexly orderable.

\begin{proof}[Proof of Theorem \ref{Thm_VCMinCO}]
 Fix $T$ a complete theory, and $M \models T$.  Suppose $T$ is VC-minimal and let $\Psi$ be a VC-minimal instantiable family witnessing this fact.  Let $\mathcal{A} = \{ \psi(M; \overline{c}) : \psi(x; \overline{z}) \in \Psi, \overline{c} \in M^{\leng(\overline{z})} \}$.  By VC-minimality, the independence dimension of $\mathcal{A}$ is $\le 1$.  Thus, by Proposition \ref{Prop_IDoneOrder}, there exists a linear order $<$ on $M$ so that all elements of $\mathcal{A}$ are either $<$-convex or the compliment of a $<$-convex set.  Fix a formula $\varphi(x; \overline{y})$.  Then, by compactness, there exists a number $K < \omega$ such that, for all $\overline{b} \in M^{\leng(\overline{y})}$, $\varphi(x; \overline{b})$ is a boolean combination of at most $K$ instances of formulas from $\Psi$.  Therefore, for any $\overline{b} \in M^{\leng(\overline{y})}$, since $\varphi(M; \overline{b})$ is a boolean combination of at most $K$ elements from $\mathcal{A}$, it is a union of at most $K+1$ $<$-convex subsets of $M$.
\end{proof}

As shown above, convex orderability is not equivalent to VC-minimality.  One interesting open question we get from this is the following: Does being the reduct of a structure with a VC-minimal theory characterize convex orderability?  Convex orderablility and VC-minimality share a similar relationship with VC-density and dp-minimality.

\begin{defn}\label{Defn_VCDensity}
 A formula $\varphi(\overline{x}; \overline{y})$ has VC-density $\le \ell$ if there exists $K < \omega$ such that, for all finite $B \subseteq \mathfrak{C}^{\leng(\overline{y})}$, we have $| S_\varphi(B) | \le K \cdot |B|^\ell$.
\end{defn}

\begin{prop}\label{Prop_ConvexlyOrderableimpliesVCDensityOne}
 If $T$ is convexly orderable, then all formulas of the form $\varphi(x; \overline{y})$ (with $x$ a singleton) have VC-density $\le 1$.
\end{prop}

\begin{proof}
 Fix $M \models T$, so $M$ is convexly orderable, and fix $<$ an ordering on $M$ witnessing this fact.  By definition, there exists $K < \omega$ such that, for all $\overline{b} \in M^{\leng(\overline{y})}$, $\varphi(M; \overline{b})$ is a union of at most $K$ $<$-convex subsets of $M$.  Therefore, there are at most $2K$ ``endpoints'' (i.e., the truth value of $\varphi(x; \overline{b})$ alternates at most $2K$ times).  So if we take any finite $B \subseteq M^{\leng(\overline{y})}$, there are at most $2K \cdot |B|$ ``endpoints'' from all the $\varphi(M; \overline{b})$ as we range over all $\overline{b} \in B$.  Therefore, we get that $|S_\varphi(B)| \le 2K \cdot |B| + 1$, hence $\varphi$ has VC-density $\le 1$.
\end{proof}

The proof of Proposition 3.2 in \cite{DGL} yields: If $T$ is a theory such that all formulas of the form $\varphi(x; \overline{y})$ have VC-density $\le 1$, then $T$ is dp-minimal.  Therefore, we get the following corollary:

\begin{cor}\label{Cor_ConvexlyOrderableimpliesDpMin}
 If $T$ is convexly orderable, then $T$ is dp-minimal.
\end{cor}

Convexly orderability does not characterize dp-minimality.  Consider the following example taken from Proposition 3.7 in \cite{DGL}:

\begin{expl}\label{Expl_dpminnotimpliesCO}
 Let $L = \{ P_n : n < \omega_1 \}$ for $P_n$ unary predicates.  For any $I, J \subset \omega_1$ finite and disjoint, let
 \[
  \sigma_{I,J} = \exists x \left( \bigwedge_{i \in I} P_i(x) \wedge \bigwedge_{j \in J} \neg P_j(x) \right)
 \]
 and let $T = \{ \sigma_{I,J} : \text{ for all such } I, J \}$.  As shown in \cite{DGL}, $T$ is complete, has quantifier elimination, and is dp-minimal.  However, we claim that $T$ has no convexly orderable structure.  Suppose, by way of contradiction, that $M \models T$ is convexly orderable, witnessed by $<$.  By pigeon-hole principle, there exists $N < \omega$ and $I \subseteq \omega_1$ with $|I| = \aleph_1$ such that $P_i(M)$ is the union of at most $N$ $<$-convex subsets of $M$ for all $i \in I$.  Hence, for each $i \in I$, $P_i(M)$ has at most $2N$ ``endpoints.''  Therefore, for any $k < \omega$ and any finite $I_0 \subseteq I$ with $|I_0| = k$, there are, at most, $2Nk+1$ $\Delta_{I_0}$-types over $\emptyset$ (where $\Delta_{I_0} = \{ P_i(x) : i \in I_0 \}$).  However, since $M \models \sigma_{I_1,I_0 - I_1}$ for all $I_1 \subseteq I_0$, there are $2^k$ $\Delta_{I_0}$-types over $\emptyset$.  Hence $2^k \le 2Nk+1$ for all $k < \omega$, a contradiction.
\end{expl}

We should remark that if we replace $L$ with $L = \{ P_n : n < \omega \}$ and build the corresponding theory $T$, then $T$ is actually VC-minimal, hence also convexly orderable.  In fact, the following is a VC-instantiable family for $T$:
\[
 \left\{ \bigwedge_{i < n} P_i(x)^{s(i)} \wedge P_n(x) : n < \omega, s \in {}^n 2 \right\}.
\]

\section{Variants on the Definition of VC-Minimality}\label{Section_VCVar}


We say a theory $T$ has \emph{low VC-density} if, for every formula $\varphi(\overline{x}; \overline{y})$, $\varphi$ has VC-density $\le \leng(\overline{x})$ (see also Section 3.2 of \cite{ADHMS}, where they call refer to this property by $\mathrm{vc}^T(n) \le n$ for all $n < \omega$).  One of the main open questions that motivates the work in this paper, asked by Aschenbrenner, Dolich, Haskell, MacPherson, and Starchenko in \cite{ADHMS}, is the following:

\textbf{Question.} If $T$ is VC-minimal, then does $T$ have low VC-density?

As shown in Proposition \ref{Prop_ConvexlyOrderableimpliesVCDensityOne}, this holds for formulas $\varphi(x; \overline{y})$ where $x$ is a singleton even when the theory $T$ is convexly orderable.  The problem comes in when one tries to show this for a formula with more than one free variable.  One means of addressing this issue is changing slightly the definition of VC-minimality.

In this section, we suggest a few modifications to the definition of VC-minimality.  We will first consider a weakening of the definition by loosening a requirement on the skeleton.  Then, we consider a strengthening of the definition by putting a constraint on the parameter variables in the skeleton.

\subsection{Weakly VC-Minimal Theories}\label{Subsection_WeaklyVCmin}


We first turn our attention to weakly VC-minimal theories.  This modification simply changes our requirement on the skeleton of the theory.  We no longer require that the entire skeleton has independence dimension $\le 1$, only that each individual formula in the skeleton does.  The main goal of this section will be to show that unstable weakly VC-minimal theories interpret an infinite linear order.

\begin{defn}\label{Defn_weakVCmin}
 We say that $T$ is \emph{weakly VC-minimal} if, for all formulas $\varphi(x; \overline{y})$ and all $\overline{b} \in \mathfrak{C}^{\leng(\overline{y})}$, there exists formulas $\psi_0(x; \overline{z}_0), ..., \psi_{n-1}(x; \overline{z}_{n-1})$ each with independence dimension $\le 1$ and $\overline{c}_0, ..., \overline{c}_{n-1}$ from $\mathfrak{C}$ of the appropriate length such that $\varphi(x; \overline{b})$ is a boolean combination of the $\psi_i(x; \overline{c}_i)$'s.
\end{defn}

If $T$ is VC-minimal and $\Phi$ is a VC-minimal instantiable family witnessing this, then, in particular, each formula in $\Phi$ has independence dimension $\le 1$ (by itself).  Hence, $T$ is also weakly VC-minimal.  To see that this is actually a strictly weaker notion, consider Example \ref{Expl_dpminnotimpliesCO} above.  This example is not VC-minimal, but it is weakly VC-minimal, since each $\varphi(x;y) = P_i(x)$ has independence dimension $\le 1$.  This also shows that weakly VC-minimal theories are not necessarily convexly orderable.

We consider a general lemma to simplify cases where we have instances of formulas equivalent to instances of formulas from a fixed set.

\begin{lem}\label{Lem_FinSubsetofFmlas}
 Let $\Psi = \{ \psi_i(x; \overline{z}_i) : i \in I \}$ be any set of formulas and let $\varphi(x; \overline{y})$ be a formula.  Suppose that, for all $\overline{b} \in \mathfrak{C}^{\leng(\overline{y})}$, there exists $i \in I$ and $\overline{c} \in \mathfrak{C}^{\leng(\overline{z}_i)}$ so that $\varphi(\mathfrak{C}; \overline{b}) = \psi_i(\mathfrak{C}; \overline{c})$.  Then, there exists a finite $I_0 \subseteq I$ such that, for all $\overline{b} \in \mathfrak{C}^{\leng(\overline{y})}$, there exists an $i \in I_0$ and a $\overline{c} \in \mathfrak{C}^{\leng(\overline{z}_i)}$ so that $\varphi(x; \overline{b})$ is equivalent to $\psi_i(x; \overline{c})$.
\end{lem}

\begin{proof}
 Use compactness on the following partial type in $\overline{y}$ over $\emptyset$ to obtain the desired result:
 \[
  \Sigma(\overline{y}) = \{ \neg \exists \overline{z}_i \forall x ( \varphi(x; \overline{y}) \leftrightarrow \psi_i(x; \overline{z}_i) ) : i \in I \}.
 \]
\end{proof}

This has the following corollary on weak VC-minimality.

\begin{cor}\label{Cor_FormulawiseWeakVCMin}
 If $T$ is weakly VC-minimal and $\varphi(x; \overline{y})$ is a formula, then there exists finitely many formulas $\Psi_0 = \{ \psi_i(x; \overline{z}_i) : i \in I_0 \}$, each a boolean combination of formulas with independence dimension $\le 1$ (in the variable $x$), such that, for any $\overline{b} \in \mathfrak{C}^{\leng(\overline{y})}$, there exists $i \in I_0$ and $\overline{c} \in \mathfrak{C}^{\leng(\overline{z}_i)}$ such that $\varphi(x; \overline{b})$ is equivalent to $\psi_i(x; \overline{c})$.
\end{cor}

\begin{proof}
 Let $\Psi$ be all boolean combinations of formulas with independence dimension $\le 1$ in the variable $x$, then use Lemma \ref{Lem_FinSubsetofFmlas} to conclude.
\end{proof}


For $\varphi(x; \overline{y})$ any formula, the relation $\le_\varphi$ defined by
\[
 \overline{b} \le_\varphi \overline{c} \text{ if and only if } \models \forall x (\varphi(x;\overline{b}) \rightarrow \varphi(x;\overline{c}))
\]
is a quasi-ordering on $\mathfrak{C}^{\leng(\overline{y})}$.
That is, $\le_\varphi$ induces a partial order on $\mathfrak{C}^{\leng(\overline{y})} / \equiv_\varphi$, where $\overline{b} \equiv_\varphi \overline{c}$ if and only if $\overline{b} \le_\varphi \overline{c}$
and $\overline{c} \le_\varphi \overline{b}$.

\begin{lem}\label{Lem_formuladichotomy}
 If $\varphi(x; \overline{y})$ has independence dimension $\le 1$ and is unstable, then $\varphi(x; \overline{y})$ has
the strict order property.  Furthermore, there is a definable subset $D \subseteq \mathfrak{C}^{\leng(\overline{y})}$ such that $\le_\varphi$ induces an
infinite linear ordering on $D / \equiv_\varphi$.
\end{lem}

\begin{proof}
 Fix any formula $\varphi(x; \overline{y})$ of independence dimension $\le 1$
that is unstable.  As it has the order property, choose $\langle c_q : q \in \mathbb{Q} \rangle$ and $\langle \overline{b}_r : r \in \mathbb{Q} \rangle$ such that $\varphi(c_q; \overline{b}_r)$ holds if
and only if $q < r$.  Now, whenever $q < r$, the sets $\varphi(\mathfrak{C}; \overline{b}_q)$
and $\varphi(\mathfrak{C}; \overline{b}_r)$ are not disjoint, do not cover $\mathfrak{C}$, and
$\varphi(\mathfrak{C}; \overline{b}_r)$ is not a subset of $\varphi(\mathfrak{C}; \overline{b}_q)$.
As $\varphi(x; \overline{y})$ has independence dimension $\le 1$, it follows
that $\varphi(\mathfrak{C}; \overline{b}_q)$ is a proper subset of $\varphi(\mathfrak{C}; \overline{b}_r)$
whenever $q < r$.  Thus, $\varphi(x; \overline{y})$ has the strict order property.

 For the second part of the lemma, take $D = \{ \overline{b} \in \mathfrak{C}^{\leng(\overline{y})} : \overline{b}_0 \le_\varphi \overline{b} \le_\varphi \overline{b}_1 \}$.  Arguing as above,
$\le_\varphi$ induces a linear ordering on the infinite set $D / \equiv_\varphi$.
\end{proof}

\begin{thm}\label{Thm_theorydichotomy}
 If $T$ is weakly VC-minimal, then either $T$ is stable or $T$ interprets an infinite linear order.
\end{thm}

\begin{proof}
 In any complete theory $T$, the set of
stable formulas is closed under boolean combinations and
$T$ is stable if and only if every formula $\varphi(x; \overline{y})$
with $\leng(x)=1$ is stable (see e.g., II.2.13 of \cite{Shelah}).
Thus, if every formula of independence
dimension $\le 1$ were stable, then every boolean combination of
such formulas would be stable.  In particular, given any formula
$\varphi(x; \overline{y})$, each of the finitely many formulas in $\Psi_0$ from
Corollary \ref{Cor_FormulawiseWeakVCMin} would be stable as well.  We argue that in this case,
$\varphi(x; \overline{y})$ cannot have the order property. Indeed, if
$\langle c_i : i < \omega \rangle$ and $\langle \overline{b}_j : j < \omega \rangle$ satisfied $\varphi(c_i; \overline{b}_j)$ if and only if $i<j$,
then replace each $\varphi(x; \overline{b}_j)$ by some $\psi(x; \overline{d}_j)$
equivalent to it, with $\psi \in \Psi_0$. The pigeon-hole principle
would then imply that some $\psi \in \Psi_0$ would have the order
property, contradicting its stability.  Thus, if every formula of
independence dimension $\le 1$ is stable, then every formula
$\varphi(x; \overline{y})$ with $\leng(x) = 1$ is stable, hence $T$ is stable.

 On the other hand, if some formula $\varphi(x; \overline{y})$ of independence
dimension $\le 1$ were unstable, then by Lemma \ref{Lem_formuladichotomy}
$\le_\varphi$ induces a linear ordering on the
infinite set $D / \equiv_\varphi$.
\end{proof}

As an application of this dichotomy, we turn to the Kueker Conjecture.  
Recall that the Kueker Conjecture asserts
that if $T$ is a complete theory in a countable language
with the property that every uncountable model is
$\aleph_0$-saturated, then $T$ is categorical in some infinite
power.

\begin{cor}\label{Cor_KuekerConj}
 If $T$ is a weakly VC-minimal theory in a
countable language such that every uncountable model is $\aleph_0$-saturated,
then $T$ is categorical in some infinite power.
\end{cor}

\begin{proof}
 In \cite{Hrushovski}, Hrushovski proves that the Kueker Conjecture
holds for (countable) stable theories and for theories admitting an infinite
 interpretable linear ordering, so the corollary follows immediately
from Theorem \ref{Thm_theorydichotomy}.
\end{proof}


What other dependent theories have this stable / infinite interpretable linear order dichotomy?  For any such theory, the Kueker Conjecture certainly holds by the above argument.

Unfortunately, we have no control over the VC-density of formulas in a weakly VC-minimal theory.

\begin{expl}\label{Expl_ManyOrdersDLO}
 Fix $n > 1$ and let $M = \mathbb{Q}^n$.  For each $i < n$, let $<_i$ be a binary relation on $M$ defined by $(x_0, ..., x_{n-1}) <_i (y_0, ..., y_{n-1})$ if and only if $x_i < y_i$.  Finally, let $T = \Th(M; <_i)_{i < n}$.  In much the same way as one shows it for $\Th(\mathbb{Q}; <)$, we can see that $T$ has quantifier elimination.  Therefore, since $x <_i y$ has independence dimension one for all $i < n$, we see that $T$ is weakly VC-minimal.  However, $T$ is not dp-minimal, witnessed by the formulas $y <_0 x <_0 z$ and $y <_1 x <_1 z$.  Moreover, one can check that the formula
 \[
  \varphi(x; y_0, ..., y_{n-1}) = \bigwedge_{i < n} x <_i y_i
 \]
 has VC-density $n$ in $T$.
\end{expl}

Therefore, there are weakly VC-minimal theories with a formula $\varphi(x; \overline{y})$ (with $x$ a singleton) with arbitrarily large VC-density.  In fact, the dp-rank of $T$ (as defined in \cite{KOU}) is at least $n$, showing that there are weakly VC-minimal theories with arbitrarily large dp-rank.  As another example, notice that Example 1.3 in \cite{KOU} is weakly VC-minimal but is not dp-minimal.  Clearly weakly VC-minimal theories are dependent.  Consider the following modification to Example \ref{Expl_ManyOrdersDLO} showing that dependence does not imply weak VC-minimality:

\begin{expl}\label{Expl_ModifiedDLO}
 Let $M = \mathbb{Q}^2$ and let $<$ be the binary relation $(x_0, x_1) < (y_0, y_1)$ if and only if $x_0 < y_0$ and $x_1 < y_1$.  Let $T^* = \Th(M; <)$ and notice that this is a reduct of the theory from Example \ref{Expl_ManyOrdersDLO} for $n = 2$, as $x < y$ if and only if $x <_0 y \wedge x <_1 y$.  Therefore, $T$ is dependent.  However, reflection along a line ($\tau_a : M \rightarrow M$ where $\tau_a(x)$ is the reflection of $x$ along the line of slope one through $a \in M$) is an automorphism of $(M; <)$.  This, coupled with quantifier elimination of $T$ from Example \ref{Expl_ManyOrdersDLO}, shows that the only independence dimension $\le 1$ formula with non-parameter variable $x$ is $x = y$ (use automorphisms to alter an instance of $\varphi$ to overlap independently with the original).
\end{expl}



\subsection{Fully VC-Minimal Theories}\label{Subsection_FullyVCmin}


In this subsection, we discuss a variant of VC-minimality we call full VC-minimality.  Our main goal will be to show that fully VC-minimal theories have low VC-density.

\begin{defn}\label{Defn_FullyVCmin}
 Say that $T$ is \emph{fully VC-minimal} if there exists a family of formulas $\Psi$ (each of which has at least $x$ as a free variable) such that:
 \begin{itemize}
  \item [(i)] The set $\mathcal{A}_\Psi = \{ \psi(\mathfrak{C}; \overline{b}) : \psi(x; \overline{y}) \in \Psi, \overline{b} \in \mathfrak{C}^{\leng(\overline{y})} \}$ has independence dimension $\le 1$ with respect to $\mathfrak{C}$, and
  \item [(ii)] for all formulas $\varphi(x; \overline{y})$, there exists $\psi(x; \overline{y})$ a boolean combination of formulas from $\Psi$ with variables $(x;\overline{y})$ such that $\varphi(x; \overline{y})$ is $T$-equivalent to $\psi(x; \overline{y})$.
 \end{itemize}
\end{defn}

\begin{lem}\label{Lem_FullyVCmin2}
 This definition is equivalent to the same definition after replacing (ii) with the following condition:
 \begin{itemize}
  \item [(ii)'] For all formulas $\varphi(x; \overline{y})$ and all $\overline{b} \in \mathfrak{C}^{\leng(\overline{y})}$, there exists $\psi_0(x; \overline{y}), ..., \psi_{n-1}(x; \overline{y}) \in \Psi$ such that $\varphi(x; \overline{b})$ is equivalent to a boolean combination of the formulas $\psi_i(x; \overline{b})$ for $i < n$.
 \end{itemize}
\end{lem}

Notice that this only differs from the definition of VC-minimal in that we insist that the instances of $\psi_i$ be exactly $\psi_i(x; \overline{b})$ for the same $\overline{b}$ from $\varphi(x; \overline{b})$.  Therefore, it is clear from this lemma that fully VC-minimal implies VC-minimal.

\begin{proof}[Proof of Lemma \ref{Lem_FullyVCmin2}]
 Since it is clear that (i) and (ii) implies (i) and (ii)', it suffices to show that (i) and (ii)' implies (i) and (ii).  So let $\Psi$ be given satisfying (i) and (ii)'.  First, let
 \[
  \Psi' = \Psi \cup \{ \delta(\overline{y}) : x \text{ is not a free variable in } \overline{y} \}.
 \]
 Since $x$ acts as a dummy variable in each $\delta(\overline{y}) \in \Psi' - \Psi$, it should be clear that $\Psi'$ still satisfies (i).  We claim that $\Psi'$ now satisfies (ii).

 Fix any formula $\varphi(x; \overline{y})$.  By (ii)' and compactness, there exists finitely many formulas $\gamma_0(x; \overline{y}), ..., \gamma_{\ell}(x; \overline{y})$ that are each a boolean combination of formulas from $\Psi'$ such that, for all $\overline{b} \in \mathfrak{C}^{\leng(\overline{b})}$, $\varphi(x; \overline{b})$ is equivalent to $\gamma_i(x; \overline{b})$ for some $i \le \ell$ (see the proof of Lemma \ref{Lem_FinSubsetofFmlas}).  Now let
 \[
  \delta_i(\overline{y}) = \forall z ( \varphi(z; \overline{y}) \leftrightarrow \gamma_i(z; \overline{y}) )
 \]
 for each $i \le \ell$ and we see that $\varphi(x; \overline{y})$ is equivalent to
 \[
  \bigwedge_{i \le \ell} \delta_i(\overline{y}) \rightarrow \gamma_i(x; \overline{y}).
 \]
 As $\delta_i(\overline{y}) \in \Psi'$, this is a boolean combination of formulas from $\Psi'$, hence showing (ii).
\end{proof}

\begin{prop}\label{Prop_WeakominFullyVCmin}
 If $T$ is weakly o-minimal, then $T$ is fully VC-minimal.
\end{prop}

\begin{proof}
 For each formula $\varphi(x; \overline{y})$ from $T$ and each $n < \omega$, we define the formula that gives the $n$th leftward ray carved out by $\varphi(x; \overline{y})$:
 \[
  \psi_{\varphi,n}(x; \overline{y}) = \neg \exists w_0 ... w_n \left(x = w_n \wedge \bigwedge_{i < n} \bigl(w_i < w_{i+1} \wedge \varphi(w_i; \overline{y}) \not\leftrightarrow \varphi(w_{i+1}; \overline{y}) \bigr) \right).
 \]
 Then, let
 \[
  \Psi = \{ \psi_{\varphi,n}(x; \overline{y}) : \varphi(x; \overline{y}) \text{ is any formula, } n \ge 1 \}.
 \]
 For all $\overline{b} \in \mathfrak{C}^{\leng(\overline{y})}$ and all formulas $\varphi(x; \overline{y})$, $\psi_{\varphi,n}(x; \overline{b})$ is a downward-closed convex set.  Therefore, we see that $\Psi$ satisfies (i).  By definition of weak o-minimality and $\psi_{\varphi,n}$, it is clear that condition (ii)' holds for any $\varphi(x; \overline{y})$.  Therefore, by Lemma \ref{Lem_FullyVCmin2}, we see that $T$ is fully VC-minimal.
\end{proof}

The converse of Proposition \ref{Prop_WeakominFullyVCmin} fails.  To see this, consider the modification of Example \ref{Expl_dpminnotimpliesCO} where we replace $\omega_1$ with $\omega$.  This theory is fully VC-minimal but is not weakly o-minimal (in fact, has no definable linear ordering).

In the remainder of this subsection, we show that fully VC-minimal theories have low VC-density.  To do this, we use methods developed from the study of uniform definability of types over finite sets (UDTFS) in \cite{Guingona2}.  We begin with the definition of UDTFS rank:

\begin{defn}\label{Defn_UDTFSRank}
 A formula $\varphi(\overline{x}; \overline{y})$ has \emph{UDTFS rank} $n < \omega$ if $n$ is minimal such that there exists finitely many formulas $\{ \psi_\ell(\overline{y}; \overline{z}_0, ..., \overline{z}_{n-1}) : \ell < L \}$ such that, for all finite non-empty $B \subseteq \mathfrak{C}^{\leng(\overline{y})}$ and for all $p(\overline{x}) \in S_\varphi(B)$, there exists $\ell < L$ and $\overline{c}_0, ..., \overline{c}_{n-1} \in B$ such that $\psi_\ell(\overline{y}; \overline{c}_0, ..., \overline{c}_{n-1})$ defines $p$ (i.e., for all $\overline{b} \in B$, $\models \psi_\ell(\overline{b}; \overline{c}_0, ..., \overline{c}_{n-1})$ if and only if $\varphi(\overline{x}; \overline{b}) \in p$).  We say that a formula has \emph{UDTFS} if it has finite UDTFS rank and we say that a theory $T$ has \emph{UDTFS} if all formulas do.
\end{defn}

We make the simple observation that, if $\varphi$ has UDTFS rank $n$, then $\varphi$ has VC-density $\le n$.  This is because the type space over any finite set $B$ is determined by some $\ell < L$ and $n$ elements from $B$, hence $| S_\varphi(B) | \le L \cdot |B|^n$.  We now show that UDTFS rank is subadditive, yielding a sufficiency of a single variable result: If all formulas of a theory $T$ of the form $\varphi(x; \overline{y})$ has UDTFS rank $\le 1$, then $T$ has low VC-density.

One should note that a very similar method was developed independently by Aschenbrenner, Dolich, Haskell, MacPherson, and Starchenko in \cite{ADHMS}.  In that paper, they use a similar technique that they call the \emph{VC $d$ property} to prove that weakly o-minimal theories have low VC-density (see Theorem 6.1 of \cite{ADHMS}).

\begin{thm}\label{Thm_UDTFSrankAdditive}
 Fix $T$ a theory and $k < \omega$.  If all formulas of the form $\varphi(x; \overline{y})$ have UDTFS rank $\le k$, then all formulas of the form $\varphi(\overline{x}; \overline{y})$ have UDTFS rank $\le k \cdot \leng(\overline{x})$.
\end{thm}

The proof of this essentially follows from the proof of Lemma 2.6 of \cite{Guingona2}.

\begin{proof}
 Fix a partitioned formula $\varphi(\overline{x}; \overline{y})$ with $n = \leng(\overline{x})$.  Suppose that
 \[
  \hat{\varphi}(x_0, ..., x_{n-2}; x_{n-1}, \overline{y}) = \varphi(\overline{x}; \overline{y})
 \]
 has UDTFS rank $\le K_0$ and we take a set of formulas
 \[
  \{ \psi_\ell(x_{n-1}, \overline{y}; w_0, \overline{z}_0, ..., w_{K_0-1}, \overline{z}_{K_0-1}) : \ell < L_0 \}
 \]
 witnessing this.  As in the proof of Lemma 2.6 of \cite{Guingona2}, for each $\ell < L_0$, let
 \[
  \psi^*_\ell(x_n; \overline{y}, \overline{z}_0, ..., \overline{z}_{k-1}) = \psi_\ell(x_{n-1}, \overline{y}; x_{n-1}, \overline{z}_0, ..., x_{n-1}, \overline{z}_{k-1})
 \]
 and suppose that $\psi^*_\ell(x_n; \overline{y}, \overline{\mathbf{z}})$ has UDTFS rank $\le K_1$ for each $\ell < L_0$.  For each $\ell < L_0$, suppose this is witnessed by
 \[
  \bigl\{ \gamma_{\ell, \ell'}(\overline{y}, \overline{z}_0, ..., \overline{z}_{K_0-1}; \overline{v}_0, \overline{u}_{0,0}, ..., \overline{u}_{0,K_0-1}, ..., \overline{v}_{K_1-1}, \overline{u}_{K_1-1,0}, ..., \overline{u}_{K_1-1,K_0-1}) : \ell' < L_1 \bigr\}.
 \]
 Again, as in the proof of Lemma 2.6 of \cite{Guingona2}, for each $\ell < L_0$ and $\ell' < L_1$, let
 \begin{align*}
  & \gamma^*_{\ell, \ell'}(\overline{y}; \overline{z}_0, ..., \overline{z}_{K_0-1}, \overline{v}_0, ..., \overline{v}_{K_1-1}) = \\
  & \gamma_{\ell, \ell'}(\overline{y}, \overline{z}_0, ..., \overline{z}_{K_0-1}; \overline{v}_0, \overline{z}_0, ..., \overline{z}_{K_0-1}, ..., \overline{v}_{K_1-1}, \overline{z}_0, ..., \overline{z}_{K_0-1}).
 \end{align*}
 Then, just as in the proof of Lemma 2.6 of \cite{Guingona2}, we see that $\{ \gamma^*_{\ell, \ell'} : \ell < L_0, \ell' < L_1 \}$ is a witness to the fact that the UDTFS rank of $\varphi$ is $\le K_0 + K_1$.

 Now, given the above information, we prove the theorem by induction on $n = \leng(\overline{x})$.  For $n = 1$, this is given by hypothesis.  Fix $n > 1$, fix $\varphi(\overline{x}; \overline{y})$ with $n = \leng(\overline{x})$.  Note that $\hat{\varphi}$ has UDTFS rank $\le k \cdot (n-1)$ by induction and each $\psi^*_\ell$ has UDTFS rank $\le k$ by assumption.  Therefore, the UDTFS rank of $\varphi$ is $\le k (n-1) + k = kn$, as desired.
\end{proof}

We now show the desired result.

\begin{thm}\label{Thm_FullyVCminVCdensityone}
 If $T$ is fully VC-minimal, then $T$ has low VC-density.
\end{thm}

\begin{proof}
 We first show that any formula of the form $\varphi(x; \overline{y})$ has UDTFS rank $\le 1$.  By Theorem \ref{Thm_UDTFSrankAdditive}, this implies $T$ has low VC-density.  Fix $\varphi(x; \overline{y})$ any such formula.  By condition (ii) of Definition \ref{Defn_FullyVCmin}, there exists $\gamma(x; \overline{y})$ a boolean combination of formulas from $\Psi$ with free variables $(x; \overline{y})$ such that $\varphi$ is equivalent to $\gamma$.  Say that
 \[
  \gamma(x; \overline{y}) = \bigvee_{\mu \in I} \left( \bigwedge_{i < n} \psi_i(x; \overline{y})^{\mu(i)} \right)
 \]
 for $n < \omega$, $I \subseteq {}^n 2$, and $\psi_0, ..., \psi_{n-1} \in \Psi$.

 We now generate a uniform algorithm for determining $\varphi$-types using only a single instance, showing $\varphi$ has UDTFS rank $\le 1$.  Fix $B \subseteq \mathfrak{C}^{\leng(\overline{y})}$ finite and $a \in \mathfrak{C}$.  Choose $i < n$, $t < 2$, and $\overline{b}_0 \in B$ such that
 \begin{itemize}
  \item [(i)] $\models \psi_i(a; \overline{b}_0)^t$, and
  \item [(ii)] $\psi_i(x; \overline{b}_0)^t$ is $\vdash$-minimal such.
 \end{itemize}
 By condition (i) of Definition \ref{Defn_FullyVCmin}, for all $j < n$ and $\overline{b} \in B$, one of the following holds:
 \begin{itemize}
  \item [(i)] $\psi_i(x; \overline{b}_0)^t \vdash \psi_j(x; \overline{b})$;
  \item [(ii)] $\psi_i(x; \overline{b}_0)^t \vdash \neg \psi_j(x; \overline{b})$;
  \item [(iii)] $\psi_j(x; \overline{b}) \vdash \psi_i(x; \overline{b}_0)^t$; or
  \item [(iv)] $\neg \psi_j(x; \overline{b}) \vdash \psi_i(x; \overline{b}_0)^t$.
 \end{itemize}
 If (i) holds, then $\models \psi_j(a; \overline{b})$ and if (ii) holds, then $\models \neg \psi_j(a; \overline{b})$, so we may assume (i) and (ii) fails.  If (iii) holds, then this contradicts the $\vdash$-minimality of $\psi_i(x; \overline{b}_0)^t$ unless $\models \neg \psi_j(a; \overline{b})$.  Similarly, if (iv) holds, then $\models \psi_j(a; \overline{b})$.  Thus, $\psi_j(a; \overline{b})$ holds if and only if
 \begin{align*}
  \delta_{i,j,t}(\overline{b}; \overline{b}_0) = & \forall x ( \psi_i(x; \overline{b}_0)^t \rightarrow \psi_j(x; \overline{b}) ) \vee \\ & [ \neg \forall x ( \psi_i(x; \overline{b}_0)^t \rightarrow \neg \psi_j(x; \overline{b}) ) \wedge \forall x ( \neg \psi_j(x; \overline{b}) \rightarrow \psi_i(x; \overline{b}_0)^t ) ]
 \end{align*}
 holds.  Finally, notice that $\gamma(a; \overline{b})$ holds if and only if
 \[
  \delta'_{i,t}(\overline{b}; \overline{b}_0) = \bigvee_{\mu \in I} \left( \bigwedge_{j < n} \delta_{i,j,t}(\overline{b}; \overline{b}_0)^{\mu(j)} \right)
 \]
 holds.  Therefore, we see that the set $\{ \delta'_{i,t}(\overline{y}; \overline{y}_0) : i < n, t < 2 \}$ shows that $\varphi$ has UDTFS rank $\le 1$.  This concludes the proof.
\end{proof}

Since weakly o-minimal theories are fully VC-minimal, Theorem \ref{Thm_FullyVCminVCdensityone} generalizes Theorem 6.1 of \cite{ADHMS}.  The problem with full VC-minimality is that, unlike standard VC-minimality, it does not generalize strong minimality.  That is, there exists strongly minimal theories that are not fully VC-minimal.  Worse than that, there are strongly minimal theories with a formula $\varphi(x; \overline{y})$ that does not have UDTFS rank $\le 1$.

\begin{expl}\label{Expl_ACF0ACVF0NoUDTFSRank1}
 In the theory $\mathrm{ACF}_0$ and the theory $\mathrm{ACVF}_{(0,0)}$, the formula $\varphi(x; \overline{y})$ given by $\varphi(x; y_0, y_1) = [x^2 + y_0 x + y_1 = 0]$ has UDTFS rank $2$.
\end{expl}

To show this, take $\alpha_0, \alpha_1, \alpha_2, \alpha_3 \in \mathbb{C}$ algebraically independent.  For each $i < j < 4$, define $\overline{b}_{i,j} = (b^0_{i,j}, b^1_{i,j})$ so that
\[
 x^2 + b^0_{i,j} x + b^1_{i,j} = (x - \alpha_i)(x - \alpha_j).
\]
Let $B = \{ \overline{b}_{i,j} : i < j < 4 \}$ and let $p(x) = \tp_\varphi(\alpha_0 / B)$.  Notice that any permutation $\sigma$ of $\{ \alpha_0, \alpha_1, \alpha_2, \alpha_3 \}$ extends to an automorphism of $\mathbb{C}$ (or $\mathbb{C}((t^{\mathbb{Q}}))$ for the $\mathrm{ACVF}_{(0,0)}$ case).  Therefore, if there was some formula witnessing UDTFS rank $\le 1$ for $p$, we can use automorphisms to derive a contradiction.  One can show explicitly that the UDTFS rank of $\varphi$ is $\le 2$, hence it is exactly $2$.

By the proof of Theorem \ref{Thm_FullyVCminVCdensityone}, $\mathrm{ACF}_0$ and $\mathrm{ACVF}_{(0,0)}$ cannot possibly be fully VC-minimal.  The property of VC-minimality has the advantage of essentially being the fusion of weakly o-minimal and strongly minimal.  This is basically why $\mathrm{ACVF}_{(0,0)}$, which has both a weakly o-minimal part (the value group) and a strongly minimal part (the residue field), is VC-minimal.  Example \ref{Expl_ACF0ACVF0NoUDTFSRank1} is disappointing because it shows that the method of Theorem \ref{Thm_UDTFSrankAdditive} does not suffice to prove that $\mathrm{ACVF}_{(0,0)}$ has low VC-density.

\section{Open Questions}


After studying convex orderability, we are left with some open questions about the property.

\textbf{Question.} If $T$ is convexly orderable, then is $T$ the reduct of a VC-minimal theory?  Is it weakly VC-minimal?

\textbf{Question.} Is there some kind of cellular decomposition for convexly orderable theories?

More generally, we have the following implications:

\begin{center}
 \begin{tabular}{c c c c c c c}
   weakly o-minimal & $\rightarrow$ & fully VC-minimal & $\rightarrow$ & low VC-density     &               &              \\
                    &               & $\downarrow$     &               &                    &  $\searrow$   &              \\
   strongly minimal & $\rightarrow$ & VC-minimal       & $\rightarrow$ & convexly orderable & $\rightarrow$ & dp-minimal   \\
                    &               &                  & $\searrow$    &                    &               & $\downarrow$ \\
                    &               &                  &               & weakly VC-minimal  & $\rightarrow$ &  dependent   \\
 \end{tabular}
\end{center}

This leaves us with open questions about missing implication arrows.

\textbf{Question.} If $T$ is dp-minimal or of low VC-density, then is $T$ a weakly VC-minimal theory?

\textbf{Question} (\cite{ADHMS})\textbf{.} If $T$ is VC-minimal, then does $T$ have low VC-density?

\nocite{*}
\bibliographystyle{plain}
\bibliography{VCMinimal}

\end{document}